\documentclass{article}

\usepackage[margin=1in]{geometry}
\usepackage{amsmath, amsfonts, amssymb, amsthm}
\usepackage{graphicx}
\usepackage{hyperref}
\usepackage{color}

\newtheorem{thm}{Theorem}[section]
\newtheorem{prop}[thm]{Proposition}
\newtheorem{cor}[thm]{Corollary}

\newtheorem{lemma}[thm]{Lemma}

\newtheorem{preremark}[thm]{Remark}
\newenvironment{remark}{\begin{preremark}\rm}{\medskip \end{preremark}}
\numberwithin{equation}{section}

\newcommand{\R}{\mathbb R}
\newcommand{\eps}{\varepsilon}

\newcommand{\grad} {\nabla}

\newcommand{\dd} {\; \mathrm{d}}

\newcommand{\ts} {\gamma}
\DeclareMathOperator{\tr} {\mathrm{tr}}

\title{Estimates on elliptic equations that hold only where the
  gradient is large}

\author{C. Imbert\footnote{CNRS, UMR 8050 \& Laboratoire d'analyse et
    de math\'ematiques appliqu\'ees, Universit\'e Paris-Est
    Cr\'eteil Val de Marne, 61 avenue du g\'en\'eral de Gaulle, 94010
    Cr\'eteil cedex, France}~ and L. Silvestre\footnote{Mathematics
    Department, University of Chicago, Chicago, Illinois 60637, USA}}

\begin{document}
\maketitle

\begin{abstract}
We consider a function which is a viscosity solution of a uniformly
elliptic equation only at those points where the gradient is large. We
prove that the H\"older estimates and the Harnack inequality, as in
the theory of Krylov and Safonov, apply to these functions.
\end{abstract}

\section{Introduction}

This paper is concerned with deriving estimates for functions
satisfying a uniformly elliptic equation only at points where the
gradient is large. For such functions, we prove a H\"older estimate
together with a Harnack inequality.

Intuitively, wherever the gradient of a function $u$ is small, the
function will be Lipschitz, so we should not need any further
information from the equation at those points in order to obtain a
H\"older regularity result. However, there is an obvious difficulty in
carrying out this proof since we do not know a priori where $|\grad
u|$ will be large and where it will be small, and these sets may be
very irregular. Moreover, the proofs of regularity for elliptic
equations involve integral quantities in the whole domain which are
hard to obtain unless the equation holds everywhere. As an extra
technical difficulty, we consider viscosity solutions which are not
even differentiable a priori.

The main contribution of this paper is the way the so-called $L^\eps$
estimate is derived. We recall that deriving an $L^\eps$ estimate
consists in getting a ``good'' estimate on the size of the superlevel
set of a non-negative super-solution that is small at least at one
point. In the uniformly elliptic case, this estimate is obtained
thanks to the pointwise Alexandrov-Bakelman-Pucci estimate. Here, we
proceed differently by estimating directly the measure of the set of
points where the super-solution can be touched by cusps from
below. This idea was inspired by \cite{cabre1997nondivergent} and
\cite{savin}, where a similar argument is carried out with paraboloids
instead of cusps. We strongly believe that a proof based on applying the ABP
estimate to the difference of the solution and a particular function, as in \cite{cc}, \cite{ks79} or \cite{safonov}, cannot be done for the result of this paper.

\paragraph{Main results.} In order to state our results, the notion of
super-solutions and sub-solutions ``for large gradients'' should be
made precise. We do it by introducing some extremal operators
depending on the ellipticity constants $\lambda$ and $\Lambda$ and
also on a parameter $\ts$ which measures how large the
gradient should be. They coincide with the classical Pucci operators (plus first
order terms) when $|\grad u| \ge \ts$, but provide no information
otherwise.  We will be dealing with merely (lower or upper)
semi-continuous functions and their gradients together with equations
will be understood in the viscosity sense.  For a $C^2$ function $u
\colon \Omega \subset \R^d \to \R$, we consider
\begin{align*}
M^+(D^2u, \grad u) &= \begin{cases}
\Lambda \tr D^2u^+ - \lambda \tr D^2u^- + \Lambda |\grad u|& \text{if } |\grad u| \ge \ts \\
+\infty &\text{otherwise,}
\end{cases} \\
M^-(D^2u, \grad u) &= \begin{cases}
\lambda \tr D^2u^+ - \Lambda \tr D^2u^- - \Lambda |\grad u|& \text{if } |\grad u| \ge \ts \\
-\infty &\text{otherwise.}
\end{cases}
\end{align*}

The main theorem of this paper is the following H\"older estimate. 
\begin{thm}[H\"older estimate] \label{t:holder}
For any continuous function $u:\overline{B_1} \to \R$ such that
\begin{align*}
& M^-(D^2u, \grad u) \leq C_0 \text{ in } B_1, \\
& M^+(D^2u, \grad u) \geq -C_0 \text{ in } B_1, \\
& \|u\|_{L^\infty(B_1)} \leq C_0,
\end{align*}
then $u \in C^\alpha(B_{1/2})$ and
\[\|u\|_{C^\alpha(B_{1/2})} \leq C C_0 \]
 where $C$ depends on $\lambda$, $\Lambda$, dimension and $\gamma /
 C_0$ and $\alpha$ depends on $\lambda$, $\Lambda$ and dimension.
\end{thm}
\begin{remark}
The constant $C$ in Theorem~\ref{t:holder} grows like $(\gamma /
 C_0)^\alpha$ as $\gamma /
 C_0$ tends to $+\infty$. That is
\[ C(d, \lambda, \Lambda, \gamma /
 C_0) = \tilde C(d, \lambda, \Lambda) \left( 1+(\gamma /C_0)^\alpha \right).\]

Note that when $\gamma=0$, then the constant $C$ becomes independent of $C_0$ and we recover the classical estimate for uniformly elliptic equations.
\end{remark}

Our second main result is the following Harnack inequality.
\begin{thm}[Harnack inequality] \label{t:harnack}
For any non-negative continuous function $u:\overline{B_1} \to \R$ such that
\begin{align*}
& M^-(D^2u, \grad u) \leq C_0 \text{ in } B_1, \\
& M^+(D^2u, \grad u) \geq -C_0 \text{ in } B_1, \\
\end{align*}
we have
\[ \sup_{B_{1/2}} u \leq C ( \inf_{B_{1/2}} u +C_0).\]
The constant $C$ depends on $\lambda$, $\Lambda$, dimension and
$\gamma/ (C_0 + \inf_{B_{1/2}} u)$.
\end{thm}

We would like to emphasize that the result stated in terms of the
extremal operators $M^+$ and $M^-$ is more general than a result which
specifies equations of a particular form. A more classical way to
write the assumption of Theorem \eqref{t:holder} would be that for
some uniformly elliptic measurable coefficients $a_{ij}(x)$, a bounded
vector field $b_j(x)$ and a bounded function $c(x)$, the function $u$
satisfies
\begin{equation} \label{e:linear-measurable}  
a_{ij}(x) \partial_{ij} u + b_i(x) \partial_i u = c(x) \ \ \text{only where } |\grad u(x)| \geq \gamma.
\end{equation}
This statement is equivalent to the assumption of our theorems if $u$
is a classical solution to the equations. Our statement with the
extremal operators $M^+$ and $M^-$ is more adequate for the viscosity
solution framework. Note also that a bounded solution to a nonlinear
equation would also satisfy our assumptions if the equation is of the
form
\[ F(D^2 u , Du, u, x) = 0,\]
and satisfies the conditions
\begin{enumerate}
\item $F(0,p,r,x) \leq C(r) |p|$ if $|p|\geq \gamma$.
\item For every fixed $p$, $r$ and $x$ such that $|p|\geq \gamma$,
  $F(A,p,r,x)$ is uniformly elliptic in $A$.
\end{enumerate}
In fact, in the case of classical solutions (or even $W^{2,d}$
solutions), this nonlinear situation is not more general than
\eqref{e:linear-measurable}, since in particular we could obtain
\eqref{e:linear-measurable} by linearizing the equation.

As mentioned above, both Theorem \ref{t:holder} and Theorem
\ref{t:harnack} derive from a so-called $L^\eps$ estimate (see
Theorem~\ref{t:Lepsilon}). Its proof is based upon a method which
seems to have originated in the work of Cabre
\cite{cabre1997nondivergent} and continued in the work of Savin
\cite{savin}. Such an idea has also been recently used in
\cite{as}. The idea is to estimate the measure of the super-level set
of super-solutions by sliding some specific functions from below and
estimating the measure of the set of contact points. In \cite{savin},
and also recently in \cite{as}, the use of the ABP estimate is
bypassed by sliding paraboloids from below. In
\cite{cabre1997nondivergent}, X. Cabre uses the distance function
squared which is a natural replacement of quadratic polynomials in a
Riemannian manifold. In \cite{as}, in order to prove the existence of
a special barrier function to their equation (see Lemma 3.3 in
\cite{as}), they slide from below a barrier to a simpler equation. In
the present paper, we slide cusps functions of the form $\varphi(x) =
-|x|^{1/2}$.

We finally mention that we chose to state and prove results for
equations with bounded (by $C_0$) right hand sides. We do so for the
sake of clarity but, as the reader can check by following proofs
attentively, it is possible to deal with continuous right hand side
$f_0$ in equations and get estimates which only depend on the
$L^d$-norm of the function $f_0$.

Our definition of $M^+$ and $M^-$ also determines the type of gradient
dependence that we allow in our equations. In terms of linear
equations with measurable coefficients as in
\eqref{e:linear-measurable}, we are assuming that $b \in L^\infty$. In
the uniformly elliptic case, the best known estimate depends only on
$\|b\|_{L^d}$, which was obtained recently in \cite{safonov}. We have
not yet analyzed whether we can extend our result to that kind of
gradient dependence. We would also like to point out that we have not
been able to obtain a satisfactory parabolic version of our results
yet.

\paragraph{Known results.} We next explain how results stated in
\cite{dfq-10,imbert,bd-10} are related to the ones presented in this
paper.

In \cite{bd-10,dfq-10}, a Harnack inequality is derived for solutions
of some singular/degenerate equations. These solutions satisfy the
assumptions of the Harnack inequality, Theorem~\ref{t:harnack}.

In \cite{imbert}, on the one hand, a Harnack inequality and H\"older
estimates are proved for functions satisfying the asumptions of this
article. Unfortunately, there is a gap in the proof of the lemma
corresponding to the $L^\eps$ estimate (see \cite[Lemma~7]{imbert}).
On the other hand, an Alexandrov-Bakelman-Pucci estimate is derived in
\cite{imbert}. The interested reader is also referred to
\cite{dfq-09,junges,cpcm} for other results for equations in
non-divergence form and \cite{acp} for equations in divergence form that are either degenerate or singular.

In \cite{delarue}, an equation of the following form is studied
\[-\tr(A(Du,u,x) D^2u) + f(x,u,Du) = 0\]
under the assumptions that
\begin{align*}
\Lambda^{-1} \lambda(p) \mathrm I \leq A(x,r,p) \leq \Lambda \lambda(p) \mathrm I, \\
|f(x,r,p)| \leq \frac 12 \Lambda (1+\lambda(p)) (1+|p|)
\end{align*}
where $\lambda(p) \geq \lambda_0 > 0$ for $|p| \geq \ts$. The main
theorem of his paper is a H\"older continuity result, which is proved
using probabilistic techniques. Note that the assumptions of our
theorems contain this situation. The most important difference between
the result in \cite{delarue} and ours is that in that paper the
equation plays some role even where $|p|$ is small, since it is
important in its proof that all the eigenvalues of $A$ are comparable
at every point.

\paragraph{Organization of the paper.} The paper is organized as
follows. In Section~\ref{s.pre}, we introduce tools that will be used
in proofs. In Section~\ref{sec:main}, we state and prove the main new
lemma. It is a measure estimate satisfied by non-negative
super-solutions. In Section~\ref{sec:leps}, we deduce a so-called
$L^\eps$ estimate from the main new lemma. In
Section~\ref{sec:holder}, a H\"older estimate is derived from the
$L^\eps$ estimate. The last section, Section~\ref{sec:harnack}, is
devoted to the proof of the Harnack inequality stated above.

\section{Preliminaries}
\label{s.pre}

\subsection{Scaling}

In this short subsection, we analyze how the equations involving
$M^\pm$ change according to scaling. Those facts will be used
repeatedly in Sections~\ref{sec:barrier} -- \ref{sec:harnack}.

If $u$ satisfies $M^+(D^2 u, \grad u) \geq A$ in $\Omega$, then $v(x)
= K u(x_0 + r x)$ satisfies the equation $M_{r,K}^+(D^2 v, \grad
v) \geq K r^2 A$ in $x_0 + r \Omega$, where
\[ M^+_{r,K}(D^2 v, \grad v) = \begin{cases}
\Lambda \tr D^2v^+ - \lambda \tr D^2v^- + r \Lambda |\grad v|&
\text{if } |\grad v| \ge rK  \ts \\
+\infty &\text{otherwise.}
\end{cases} \]
Note that if $r \leq 1$ and $K \ge 1$ then $M^+_{r,K} \leq
M^+$. Therefore, in particular, $M^+(D^2 v, \grad v) \geq r^2 K A$ in
$x_0 + r \Omega$.

Likewise, if $M^-(D^2 u, \grad u) \leq A$ in $\Omega$, then $M^-(D^2
v, \grad v) \leq r^2 K A$ in $x_0 + r\Omega$.

\subsection{The growing ink-spots lemma} 

In this section, we state and prove a consequence of Vitali's covering
lemma. This result replaces the usual Cald\'eron-Zygmund decomposition
\cite{cc} in order to derive a so-called $L^\eps$ estimate from
Corollary~\ref{c:weakweakharnack}. It is a statement from measure theory which is essentially the same that was used in the original work by Krylov and Safonov \cite{ks79}. The suggestive name \emph{growing} (or \emph{crawling}) \emph{ink spots} was nailed by E. M. Landis according to \cite{ks80}.

\begin{lemma}[Growing ink-spots lemma] \label{t:inkspots-vitali}
Let $E \subset F \subset B_1$ be two open sets. We make the
following two assumptions for some constant $\delta \in (0,1)$.
\begin{itemize}
\item If any ball $B \subset B_1$ satisfies $|B \cap E| > (1-\delta)
  |B|$, then $B \subset F$.
\item $|E| \leq (1-\delta) |B_1|$. 
\end{itemize}
Then $|E| \leq (1-c\delta) |F|$ for some constant $c$ depending on
dimension only.
\end{lemma}

\begin{proof}
For every $x \in F$, since $F$ is open, there exists some maximal ball which is contained in $F$ and contains $x$. We choose one of those balls for each $x \in F$ and call it $B^x$.

If $B^x = B_1$ for any $x \in F$, then the result of the theorem follows immediately since $|E| \leq (1-\delta) |B_1|$, so let us assume that it is not the case.

We claim that $|B^x \cap E| \leq (1-\delta)|B^x|$. Otherwise, we could find a slightly larger ball $\tilde B$ containing $B^x$ such that $|\tilde B \cap E| > (1-\delta) |\tilde B|$ and $\tilde B \not\subset F$, contradicting the first hypothesis.

The family of balls $B^x$ covers the set $F$. By the \emph{Vitali covering lemma}, we can select a finite subcollection of non overlapping balls $B_j := B^{x_j}$ such that $F \subset \bigcup_{j=1}^K 5B_j$.

By construction, $B_j \subset F$ and $|B_j \cap E| \leq (1-\delta) |B_j|$. Thus, we have that $|B_j \cap F \setminus E| \geq \delta |B|$. Therefore
\begin{align*}
|F \setminus E| &\geq \sum_{j=1}^K |B_j \cap F \setminus E| \\
& \geq \sum_{j=1}^K \delta |B_j| \\
& = \frac{\delta}{5^d} \sum_{j=1}^K |5 B_j|  \geq \frac{\delta}{5^d} |F|.
\end{align*}
The proof is finished with $c = 1/5^d$.
\end{proof}

\section{Main new lemma}
\label{sec:main}

The lemma in this section is the main difference with the classical
case. It is the only lemma whose proof differs substantially with the
uniformly elliptic case ($\ts = 0$).

\begin{lemma}[A measure estimate] \label{l:measure-estimate}
There exist two small constants $\eps_0>0$ and $\delta>0$, and a
large constant $M>0$, so that if $\ts \leq \eps_0$, for any lower
semi-continuous function $u \colon B_1 \to \R$ such that
\begin{align*}
& u \geq 0 \text{ in } B_1, \\
& M^-(D^2u, \grad u) \leq 1 \text{ in } B_1, \\
& | \{u > M\} \cap B_1 | > (1-\delta) |B_1|, 
\end{align*}
then $u > 1$ in $B_{1/4}$.
\end{lemma}
\begin{remark}
Amusingly enough, the values of $M$ and $\eps_0$ in the lemma above
are absolute constants. They do not depend on $\lambda$, $\Lambda$ or
the dimension. But the constant $\delta$ does.
\end{remark}

\subsection{The proof for classical solutions}

The proof of Lemma \ref{l:measure-estimate} is easier to understand
when $u$ is a smooth function. We will first describe the proof in
this case. In the next subsection we will explain why the result holds
for lower semi-continuous viscosity solutions in general.

\begin{prop}\label{p:sc-smooth}
Lemma \ref{l:measure-estimate} holds if $u$ is a $C^2$ super-solution.
\end{prop}

\begin{proof}
We do the proof by contradiction. Assume that for all $\eps_0$,
$\delta$, $M$, we can find $u$ as above and such that $u(x_0) \leq 1$
for some point $x_0 \in B_{1/4}$.

Consider $U = \{ u> M\}\cap B_{1/4}$. For every $x \in U$, let $y \in
\overline {B_1}$ be a point where the minimum of $u(y) +
10|y-x|^{1/2}$ is achieved.

On one hand, since $u \geq 0$ in $B_1$ and $x \in U \subset B_{1/4}$,
then $u(z) + 10|z-x|^{1/2} > 5\sqrt{3}$ if $z \in \partial B_1$. On
the other hand, $u(x_0) + 10|x_0-x|^{1/2} \leq 1 + 5\sqrt{2} < 5\sqrt
3$. Therefore, the minimum will never be achieved on the boundary and
$y \in B_1$. Moreover, we obtain that $u(y) + 10|y-x|^{1/2} \leq 1 +
5\sqrt{2}$ and in particular $u(y) \leq 1 + 5\sqrt{2}$.

We choose the constant $M$ in this lemma to be $M := 2 + 5\sqrt{2}$
(note that $M$ does not depend on anything!). In this way, we know
that $u(y) < M$. In particular $x \neq y$ and $|z-x|^{1/2}$ is
differentiable at $z=x$. 

Note that for one value of $x$, there could be more than one point $y$
where the minimum is achieved. However, the value of $y$ determines
$x$ completely since we must have
\[
\grad u(y) = 5 (x-y) |y-x|^{-3/2}.
\]
For convenience, let us call $\varphi(z) = -10|z|^{1/2}$. We thus have
\begin{eqnarray} \label{e:equalgrad}
\grad u(y) = \grad \varphi(y-x), \\
 \label{e:ineqhessian}
D^2 u(y) \geq D^2 \varphi(y-x).
\end{eqnarray}

The expressions \eqref{e:equalgrad} and \eqref{e:ineqhessian},
together with the equation $M^-(D^2 u, Du) \leq 1$, implies that
\begin{equation} \label{e:boundedhessian}
|D^2 u(y)| \leq C\left( 1 + |D^2 \varphi(y-x)| + |\grad \varphi(y-x)| \right)
\end{equation}
provided that $\eps_0 \le \min_{B_{5/4}} |\grad \varphi| = 2 \sqrt{5}$. In the previous inequality,
$C$ depends on ellipticity constants and dimension. 

Since for each value of $y$, there is only one value of $x$, we can
define a map $m(y) := x$. Let us call $\mathcal T$ the domain of the
map $m$. That is $\mathcal T$ is the set of values that $y$ takes as
$x \in U$. We know that $\mathcal T \subset \{y:u(y) < M\}$ and that
$m(\mathcal T) = U$.

Replacing $x=m(y)$ in \eqref{e:equalgrad} and applying the chain rule,
we obtain
\[ D^2 u(y) = D^2 \varphi(y-m(y)) \ \left( I - Dm(y) \right).\]

Solving for $Dm$ and using the estimate \eqref{e:boundedhessian}, we
get (in terms of Frobenius norms)
\[ | Dm(y)| \leq 1 + C \frac{  1 + |D^2 \varphi(y-x)| + |\varphi(y-x)| }{| D^2 \varphi(y-x)|} \leq C.\]
Therefore
\[ (1-4^d \delta)|B_{1/4}| \leq |U| = \int_{\mathcal T} |\det Dm(y)| \dd y \leq C |\mathcal T|.\]

Since, $\mathcal T \subset \{y : u(y) < M\}$, from our assumptions we
have that $|\mathcal T| \leq \delta |B_1|$. This is a contradicion if
$\delta$ is small enough (depending on ellipticity constants and
dimension). The proof is now complete.
\end{proof}

\subsection{Formalizing the proof for viscosity solutions}

In this subsection, we explain how to derive
Lemma~\ref{l:measure-estimate} for merely lower semi-continuous
viscosity super-solutions. In order to do so, we use classical
inf-convolution techniques to reduce to the case of semi-concave
viscosity super-solutions (Proposition~\ref{p:sc-enough}).  We then
prove Lemma~\ref{l:measure-estimate} in the semi-concave case
(Proposition~\ref{p:sc} below).

\begin{prop}\label{p:sc-enough}
Assume Lemma~\ref{l:measure-estimate} is proved for semi-concave
super-solutions. Then the lemma is also true for a lower
semi-continuous super-solution $u$.
\end{prop}
\begin{proof}
Let us consider a merely lower semi-continuous super-solution $u$
defined in $B_1$.

Let $v := \min (u,2M)$ where $M$ is given by
Lemma~\ref{l:measure-estimate} for semi-concave solutions.  Note that
$v$ is still a super-solution because it is the minimum between two
super-solutions. We have $0 \leq v \leq 2M$.

Consider the inf-convolution of $v$ of parameter $\eps>0$:
\[ v_\eps(x) = \inf_{y \in B_1} (v(y) + (2\eps)^{-1}|y-x|^2).  \]
It is classical to prove that $v_\eps$ is still a super-solution at $x
\in B_{1-\delta}$ (for $\delta>0$) of the same equation provided that
we can show that $y_x \notin B_1$.

Consider $y_x \in \overline{B_1}$ such that 
\[ v_\eps (x) = v(y_x) + (2\eps)^{-1} |y_x-x|^2 \leq v (x).\]
Then 
\[ |y_x-x| \leq 2\sqrt{\|v\|_\infty \eps } = 2\sqrt{2M\eps}.\]
Thus, for any $\delta>0$, $v_\eps$ is a super-solution in
$B_{1-\delta}$ provided that $2\sqrt{2M\eps} < \delta$.

Note that $v_\eps$ is semi-concave and 
\[ D^2 v_\eps \leq \eps^{-1} I.\]

Since $v$ is lower semicontinuous, it is classical to show that
$v_\eps$ converges to $v$ in the half relaxed sense (which is exactly
the same as $\Gamma$-convergence). Moreover,
\[ \{u > M\} = \bigcup_{\eps > 0} \{v_\eps > M\}.\]
Note that as $\eps \to 0$, the sets $\{v_\eps > M\}$ is an increasing
nested collection, therefore
\[ |\{ u > M \}| = \lim_{\eps \to 0} |\{v_\eps > M\}|.\]

For $\eps$ sufficiently small, we can apply
Lemma~\ref{l:measure-estimate} (appropriately scaled to the ball
$B_{1-\delta}$ instead of $B_1$) and obtain that $v_\eps \geq 1 \text{
  in } B_{(1-\delta)/4}$. Since $u \geq v_\eps$ and $\delta$ is
arbitrarly small, the proof is finished.
\end{proof}

\begin{prop}\label{p:sc}
Lemma \ref{l:measure-estimate} holds if $u$ is a semi-concave
viscosity super-solution.
\end{prop}

\begin{proof}
The main idea of the proof was already explained in Lemma \ref{l:measure-estimate} for $u \in C^2$. Here we need to work harder in order to deal with the technical difficulty that we do not assume the function $u$ to be second differentiable. Yet, the proof follows essentially the same lines.

In order to organize the proof. We list the main steps in bold letters.

We assume that we have a semi-concave function $u$, which satisfies
\[ u \geq 0 \qquad \text{ and } \qquad M^-(D^2 u, \grad u) \leq 1 \text{ in } B_1.\]

We assume also that
\begin{equation} \label{e:contradicting-hypothesis}  \min_{B_{1/4}} u \leq 1 \quad \text{ and } \quad |\{ u > M \} \cap B_1| >
(1-\delta)|B_1|
\end{equation} 
in order to obtain a contradiction.

\paragraph{Step 0. Analyzing the semi-concavity assumption.} \hfill

We assume only that $D^2 u \leq C_0$ in the
sense that $u(x) - C_0 |x|^2/2$ is concave. This means that for every point $x_0 \in B_1$ there exists a vector $p \in \R^d$ (a vector in the super-differential), which is $p=\grad u(x_0)$ in case $u$ is differentiable at $x_0$, so that
\begin{equation} \label{e:semi-concavity}  u(x) \leq u(x_0) + p \cdot (x-x_0) + \frac{C_0}2 |x-x_0|^2.
\end{equation}
for all $x \in B_1$.

We finally recall that by Alexandrov theorem, the semi-concave function $u$ is pointwise second differentiable almost everywhere. That means that there exists a set of measure zero $E \subset B_1$, so that at every point $x \in B_1 \setminus E$, the function $u$ is differentiable and there exists a symmetric matrix $D^2u(x)$ such that
\[ u(y) = u(x) + (y-x) \cdot \grad u(x) + \frac 12 \langle D^2u(x) \, (y-x), (y-x) \rangle + o(|x-y|^2).\]
Moreover, we also have \cite{jbhu}
\[ \grad u(y) = \grad u(x) + D^2u(x) \, (y-x) + o(|x-y|),\]
where by $\grad u(y)$ we mean any vector in the super-differential of $u$ at $y$.

\paragraph{Step 1. Touching $u$ with cusps from below.} \hfill

As in the proof for $u \in C^2$, we define $\varphi(x) =
-10|x|^{1/2}$ and $M = 2+5\sqrt 2$. 

Consider the open set $U = \{u > M\} \cap B_{1/4}$. From our assumption \eqref{e:contradicting-hypothesis}, we have that $|U| > |B_{1/4}| - \delta |B_1|$, which is a significant measure for $\delta$ small. We can assume for example that $|U| \geq |B_{1/8}|$, which is a constant which depends on the dimension $d$ only.

For every $x \in U$, we look for the point $y \in B_1$ which realizes the following minimum.
\begin{equation} \label{e:def-of-y}  u(y) - \varphi(y-x) = \min\{u(z) - \varphi(z-x) : z \in \overline{B_1} \}.
\end{equation}
Equivalently, if we let $q(x)= \min_{z \in \overline{B_1}} (u (z) - \varphi(z-x))$, we have
\begin{equation} \label{e:alt-def-of-y}  \begin{aligned}
u(y) &= \varphi(y-x)+q(x), \\
u(z) &\geq \varphi(z-x)+q(x) \quad \forall z \in B_1.
\end{aligned}
\end{equation}

Since $\min_{B_{1/4}} u \leq 1$, we observe that $q(x) \leq 1 - \min_{B_{1/2}} \varphi = 1 + 5 \sqrt 2$. Consequently, $y \notin \partial B_1$, since for those values of $y$ we would have $\varphi(y-x) + q(x) < 0 \leq u(y)$. Moreover, $u(y) = \varphi(y-x) + q(x) \leq 1+5 \sqrt 2 = M-1$. In particular $y \notin U$ and $y \neq x$.

Since $u$ is a semi-concave function, at the point $y$ where it is touched by below by the smooth function $\varphi$, it must be differentiable and $\grad u(y) = \grad \varphi(y-x)$. Further analysis on the second derivatives of $u$ at $y$ is postponed to later in the proof.

\paragraph{Step 2. Defining the contact set  $\mathcal T$.} \hfill

We define $\mathcal T$ as the set of contact points $y \in \overline{B}_1$ for all values of $x \in U$. In other words, for any $y \in \mathcal T$,
there exists $x_y \in U$  such that \eqref{e:def-of-y} holds. This definition is just a rephrasing of the definition of $\mathcal T$ given in the proof of Lemma~\ref{l:measure-estimate}.

As it was mentioned above, we have $u(y) \leq M-1$ for all $y \in \mathcal T$. Thus 
\[ \mathcal T \subset B_1 \cap \{ u \le M-1\}.\]

\paragraph{Step 3. $\grad u$ is Lipschitz on $\mathcal T$.} \hfill

Since $u(x) > M$ for all $x \in U$ and $u(y) \leq M-1$ for all $y \in \mathcal T$, then we must have $|y-x| > \eps$ for some $\eps>0$ depending on the modulus of continuity of $u$. The function $\varphi$ has a singularity at the origin. This constant $\eps>0$ tells us that we are evaluating $\varphi(y-x)$ away from this singularity where $\varphi$ is $C^2$ and $|D^2 \varphi| < C \eps^{-3/2}$.

Let $x_1, y_1$ and $x_2, y_2$ be two pairs of corresponding points
(they are two pairs of $x,y$ points satisfying
\eqref{e:def-of-y}). Let $r = 2|y_1-y_2|$. For any $z \in B_r(y_1)$,
we use the bound of $D^2 \varphi$ above and \eqref{e:alt-def-of-y}, to
obtain
\begin{align*}
u(z) &\geq \varphi(z-x_1) \\
& \geq \varphi(y_1-x_1) + \grad \varphi(y_1-x_1) \cdot (z-y_1) - C \eps^{-3/2} r^2, \\
& = u(y_1) + \grad u(y_1) \cdot (z-y_1) - C \eps^{-3/2} r^2.
\end{align*}
In particular, for $z = y_2$, 
\[ u(y_2) \geq u(y_1) + \grad u(y_1) \cdot (y_2-y_1) - C \eps^{-3/2} r^2.\]
Exchanging the roles of $y_1$ and $y_2$, we also get
\[u(y_1) \geq u(y_2) + \grad u(y_2) \cdot (y_1-y_2) - C \eps^{-3/2} r^2.\]

Replacing this bound for the value of $u(y_1)$ in the first inequality, we get
\[ u(z) \geq u(y_2) + \grad u(y_2) \cdot (y_1-y_2) + \grad u(y_1) \cdot (z-y_1) - C \eps^{-3/2} r^2.\]

Moreover, from \eqref{e:semi-concavity}, we also have
\[ u(z) \leq u(y_2) + \grad u(y_2) \cdot (z-y_2) + Cr^2.\]

Subtracting the two inequalities above, we obtain
\[ (\grad u(y_1) - \grad u(y_2)) \cdot (z-y_1) \leq C (\eps^{-3/2}+1) r^2.\]

Since $z$ is an arbitrary point in $B_r(y_1)$, we conclude that
$|\grad u(y_1) - \grad u(y_2)| \leq C (1+\eps^{-3/2}) r$. That is, we
proved that $\grad u$ is Lipschitz on $\mathcal T$. The estimate of
the Lipschitz norm $[\grad u]_{Lip(\mathcal T)}$ that we obtained
depends on $\eps$ and consequently on the modulus of continuity of
$u$. It is not a universal constant.

\paragraph{Step 4. The map $m: \mathcal T  \to U$.} \hfill

As it was pointed out above, $u$ must be differentiable at the point
$y$ and $\grad u(y) = \grad \varphi(y-x)$. Note that the value of
$\nabla \varphi(y-x) = -5 |y-x|^{-3/2} (y-x)$ determines
uniquely the value of $(y-x)$. In particular, for every $y \in
\mathcal T$, there is a unique $x \in U$ so that \eqref{e:def-of-y}
holds, and that is the point $x$ such that $\grad u(y) = \grad
\varphi(y-x)$. Let us define $m : \mathcal T \to U$ as the function
that maps $y$ into $x$. That is, from the implicit definition
\begin{equation} \label{e:implicit}  \grad u(y) = \grad \varphi(y-m(y)),
\end{equation}
we deduce
\[ m(y) = y - (\grad \varphi)^{-1} \grad u(y),\]
where by $(\grad \varphi)^{-1}$ we mean the inverse of $\grad \varphi$
as a function from $\R^d \to \R^d$.

We showed already that $\grad u$ is Lipschitz on $\mathcal
T$. Clearly, the map $(\grad \varphi)^{-1}$ which maps $\grad
\varphi(y-x)$ into $y-x$ has a singularity for large gradients, or
equivalently where $y-x$ is close to the origin. As it was pointed out
above, we always have $|y-x| > \eps$ for some $\eps >0$ depending on
the modulus of continuity of $u$. So at least we know that on the set
$\mathcal T$, $(\grad \varphi)^{-1}$ will be a Lipschitz map (in fact
smooth) with a Lipschitz constant depending on $\eps$ (and
consequently on the modulus of continuity of $u$). This implies that
$m$ is Lipschitz. Therefore, $m$ is differentiable almost everywhere
and the classical formula holds
\begin{equation} \label{e:jacobian}  |U| = \int_{\mathcal T} |\det Dm(y)| \dd y.
\end{equation}

\paragraph{Step 5. A universal estimate on $Dm$.} \hfill

So far we have only estimated $|Dm(y)|$ in terms of $\eps$. This was
only a technical step to justify writing the expression
\eqref{e:jacobian}. Now we will obtain an estimate for $|Dm(y)|$
depending only on the universal constants $\lambda$, $\Lambda$ and
$d$.

As we mentioned in step 0, the function $u$ is pointwise second
differentiable except for a set $E$ of measure zero. In particular,
for all $y \in \mathcal T \setminus E$, we have $M^-(D^2 u(y), \grad
u(y)) \leq 1$ in the classical sense and we can do the following
computations below.

We take $\gamma$ sufficiently small in order to ensure that $|\grad
u(y)| = |\grad \varphi(y-x)| > \gamma$ for all $y \in B_1$ and $x \in
B_{1/4}$. Thus, the equation $M^-(D^2u(y), Du(y)) \leq 1$ is
meaningful and we obtain
\begin{equation} \label{e:the-equation} 
\lambda \tr (D^2 u(y))^+ - \Lambda \tr (D^2 u(y))^- = M^-(D^2 u(y),
\grad u(y)) + \Lambda |\nabla u(y)| \leq C(1 + |y-x|^{-1/2}).
\end{equation}

Moreover, from \eqref{e:alt-def-of-y}, we have that $D^2 u(y) \geq D^2
\varphi(y-x)$. In particular the negative part of the Hessian of
$\varphi$ controls the Hessian of $u$: $(D^2 u(y))^- \leq (D^2
\varphi(y-x))^-$. Combining this fact with \eqref{e:the-equation} we
obtain
\[ |D^2 u(y)| \leq C \left( (D^2 \varphi(y-x))^- + 1+|y-x|^{-1/2}
\right) \leq C (1+|y-x|^{-3/2}).\]

We now differentiate \eqref{e:implicit} (recall that this is a valid
computation for $y \in \mathcal T \setminus E$) and obtain
\begin{equation}\label{eq:differentiate}
 D^2 u(y) = D^2 \varphi(y-x) (I - Dm(y)). 
\end{equation}
Therefore,
\begin{align*}
|Dm(y)| &= D^2 \varphi(y-x)^{-1} \left( D^2 \varphi(y-x) - D^2 u(y) \right)\\
&\leq \|D^2 \varphi(y-x)^{-1}\| \ \|D^2 \varphi(y-x) - D^2 u(y)\|,\\
&\leq C
\end{align*}
where $C$ is a universal constant. For the last inequality we used
that $\|D^2 \varphi(y-x)^{-1}\| = C |x-y|^{3/2}$ and $\|D^2
\varphi(y-x) - D^2 u(y)\| \le C (1+ |x-y|^{-3/2})$. Note how the dependence
on $|x-y|$ cancels out. This step would not work for some other
choices of $\varphi$, for example $\varphi(x) = -|x|$.

Thus, we obtained that $|Dm| \leq C$ almost everywhere in $\mathcal
T$, for a universal constant $C$. We can replace this estimate in
\eqref{e:jacobian} and obtain
\[ |U| \leq \int_{\mathcal T} C^d \dd y = C^d |\mathcal T|.\]

This gives us a lower bound for the measure of the set of contact
points $\mathcal T$. Thus, $|\mathcal T| \geq \delta |B_1|$ for some
$\delta>0$. Since $\mathcal T \subset \{u \leq M-1\}$, we obtain the
contradiction with \eqref{e:contradicting-hypothesis} and finish the
proof.
\end{proof}

\begin{remark}
Note that in order for the formula \eqref{e:jacobian} to hold, we need
$\grad \Gamma$ to be Lipschitz in $\mathcal T$. This may be the most
important difficulty of the viscosity solution adaptation of the
argument.

For example, just the fact that $\Gamma$ is concave implies that
$D^2\Gamma$ is well defined almost everywhere, but yet that does not
imply the formula \eqref{e:jacobian} since a cone functions
$\Gamma(x)=|x|$ would be a counterexample for $\mathcal T = \{0\}$ if
we understand $\grad \Gamma$ in the sense of subdifferentials.

In this case we also have $\Gamma$ differentiable in $\mathcal
T$. Moreover, it is not hard to see that $\Gamma$ is pointwise
$C^{1,1}$ at every point $y \in \mathcal T$ in the sense that there
exists a constant $C$ and $b \in \R^d$ such that $-C|h|^2 \leq
\Gamma(y+h) - \Gamma(y) - b \cdot h \leq C|h|^2$. These facts allow us
to differentiate \eqref{e:implicit} and get
\eqref{eq:differentiate}. 
\end{remark}

\section{A barrier function and the doubling property}
\label{sec:barrier}

Consider the barrier function $b(x) = |x|^{-p}$. Assume initially that
$\ts=0$. We compute, for $x \in B_2 \setminus \{0\}$,
\[
\begin{aligned}
M^-(D^2 b, \grad b) &= \lambda p(p+1) |x|^{-p-2} - \Lambda (d-1) p |x|^{-p-2} - \Lambda p |x|^{-p-1} \\
&= p |x|^{-p-2} \left( \lambda (p+1) - \Lambda (d-1) - \Lambda |x| \right)\\
&\geq p |x|^{-p-2} \left( \lambda (p+1) - \Lambda (d+1) \right)\\&
\geq p |x|^{-p-2}
 \ \ \text{ if $p$ is large enough.}
\end{aligned}
\]
Thus, the function $b(x) = |x|^{-p}$ is a sub-solution of the Pucci
equation $M^-(D^2 b,\nabla b) \geq 0$ in $B_2 \setminus \{0\}$ with
$\gamma =0$. Likewise, it will be a sub-solution of $M^-(D^2 b, \grad
b)~\geq~0$ in $B_2 \setminus \{0\}$ provided that $\ts$ is chosen
smaller than the minimum norm of its gradient.

Using this barrier function, we prove the following doubling property
for lower bounds of super-solutions.

\begin{lemma}[Doubling property for super-solutions] \label{l:doubling}
There exists a small constant $\eps_0>0$ depending on
$\lambda,\Lambda$ and dimension such that if $u \geq 0$ is a
super-solution $M^-(D^2 u, \grad u) \leq 1$ in $B_2$ and $u > M$ in
$B_{1/4}$ for some large constant $M$, then $u > 1$ in $B_1$.
\end{lemma}
\begin{remark}
The constant $M$ depends on $\lambda$, $\Lambda$, $\gamma$ and dimension.
\end{remark}
\begin{proof}
We compare the function $u$ with 
\[ B(x):= M \frac{(|x|^{-p} - 2^{-p})}{2 \cdot 4^p} .\]
We choose $M \ge 1$ sufficiently large so that $B \geq 1$ in $B_1$ and
$|\grad B| \geq \gamma$ in $B_1$.

We have 
\begin{align*}
M^-(D^2 B, \grad B) & \ge \frac{M}{2 \cdot 4^p} M^-(D^2b,\nabla b) \\
& \ge \frac{M}{2 \cdot 4^p} p 2^{-p-2} \\
& \ge 2 \quad \text{ for $M$ large enough.}
\end{align*}

 Moreover, $B = 0$ on $\partial B_2$ and $B < M$ in $\partial
 B_{1/4}$. Therefore $B \leq u$ in the ring $B_2 \setminus B_{1/4}$
 (this is the comparison principle between the viscosity
 super-solution $u$ and the classical sub-solution $B$, which follows
 directly from the definition of viscosity solution).

Therefore, we have $u \geq B \geq 1$ in $B_1$. Moreover, for $\eps = \min_{B_{1/4}} (u/M-1)$ we also have $u \geq (1+\eps) M > 1$ in $B_1$, which finishes the proof.
\end{proof}

Combining Lemmas~\ref{l:measure-estimate} with \ref{l:doubling}, we obtain the following corollary
\begin{cor} \label{c:weakweakharnack}
There exist small constants $\eps_0>0$ and $\delta>0$, and a large
constant $M>0$, so that if $\ts \leq \eps_0$, for any continuous
function $u:B_2 \to \R$ such that
\begin{align*}
& u \geq 0 \text{ in } B_2, \\
& M^-(D^2u, \grad u) \leq 1 \text{ in } B_2, \\
& | \{u > M\} \cap B_1 | > (1-\delta) |B_1|, 
\end{align*}
then $u > 1$ in $B_1$.
\end{cor}
\begin{remark}
Note that the constant $M$ in Corollary \ref{c:weakweakharnack} is the
product of the two constants $M$ in Lemma \ref{l:measure-estimate} and
Lemma \ref{l:doubling}.
\end{remark}
\begin{proof}
Let $M_1$ be the constant from Lemma~\ref{l:measure-estimate} and
$M_2$ be the one coming from Lemma~\ref{l:doubling}. Then the function
$v = u / M_2$ satisfies the assumption of
Lemma~\ref{l:measure-estimate} for $M_2 \ge 1$ (which can be assumed
without loss of generality). We conclude that $v > 1$ in $B_{1/4}$,
\textit{i.e.} $u > M_2$ in $B_{1/4}$. We then can apply
Lemma~\ref{l:doubling} and get $u > 1$ in $B_1$. 
\end{proof}

The following corollary is just a scaled version of the above result.

\begin{cor} \label{c:scaled-weakweakharnack}
There exists small constants $\eps_0>0$ and $\delta>0$, and a large
constant $M>0$, so that if $\ts \leq \eps_0$, for any $r \leq 1$,
$\kappa \geq 1$ and a continuous function $u:\overline{B_r} \to \R$
such that
\begin{align*}
& u \geq 0 \text{ in } B_r, \\
& M^-(D^2u, \grad u) \leq \kappa \text{ in } B_r, \\
& | \{u > \kappa M\} \cap B_{r/2} | > (1-\delta) |B_{r/2}|, 
\end{align*}
then $u > \kappa$ in $B_{r/2}$.
\end{cor}

\begin{proof}
The scaled function $u_r(x) = u(rx/2)/\kappa$ satisfies the scaled equation
\[ M^-_{r/2,\kappa}(D^2u_r, \grad u_r) \leq r^2 \leq 1 \text{ in } B_2. \]
 We remark that it satisfies a stronger equation since $\ts$ can be replaced
 by the smaller value $\kappa^{-1} r \gamma$. So we can apply Corollary \ref{c:weakweakharnack} to $u_r$ and obtain the result.
\end{proof}

\section{The $L^\eps$ estimate}
\label{sec:leps}

Combining Corollary \ref{c:weakweakharnack} with Lemma~\ref{t:inkspots-vitali}, we obtain the $L^\eps$ estimate.

\begin{thm}[$L^\eps$ estimate] \label{t:Lepsilon}
There exists  small constants $\eps_0>0$ and $\eps>0$, so that if $\ts
\leq \eps_0$, for any lower semi-continuous function $u:B_2 \to \R$ such that
\begin{align*}
& u \geq 0 \text{ in } B_2, \\
& M^-(D^2u, \grad u) \leq 1 \text{ in } B_2, \\
& \inf_{B_1} u \leq 1,
\end{align*}
then 
\[ |\{ u > t \} \cap B_1 | \leq C t^{-\eps}\]
for all $t>0$.
\end{thm}
\begin{remark}
This estimate is referred to as the $L^\eps$ estimate since it yields
an estimate on $\int_{B_1} u^\eps (x) \, dx$ (depending on $C$ only). 
\end{remark}
\begin{proof}
In order to prove the result, we will prove the equivalent expression
\[ |\{u > M^k\} \cap B_1| \leq \tilde{C} M^{-\eps k} \]
where $M$ is the constant from Corollary
\ref{c:scaled-weakweakharnack} and $\eps>0$ has to be properly chosen.

Let $A_k := \{u > M^k\} \cap B_1$, which are open sets. Since $\inf_{B_1} u \leq 1$, from Corollary \ref{c:weakweakharnack}, $|A_1| \leq (1-\delta)|B_1|$. Since $A_k \subset A_1$ for all $k>1$, then we also have $|A_k| \leq (1-\delta)|B_1|$ for all $k$.

We note that Corollary \ref{c:scaled-weakweakharnack}, with $\kappa=M^k$, says that every time a ball $B \subset B_1$ satisfies 
$|B \cap A_{k+1}| > (1-\delta) |B|$, then $B \subset A_{k}$. Using Lemma\ref{t:inkspots-vitali}, we obtain
\[ |A_{k+1}| \leq (1-c\delta) |A_k|, \]
and therefore, by induction, $|A_k| \leq (1-c\delta)^{k-1} (1-\delta) |B_1| = \tilde{C} M^{-\eps k}$, where $-\eps = \log(1-c\delta) / \log M$ and $\tilde{C}=(1-c\delta)^{-1}(1-\delta)|B_1|$.

This finishes the proof.
\end{proof}

The following lemma is a scaled version of Theorem \ref{t:Lepsilon}.
\begin{lemma}[scaled $L^\eps$ estimate] \label{l:improvementOfOscillation}
There exists small constants $\tilde \eps_0>0$, $\eps_1>0$ and
$\theta>0$, so that if $\ts \leq \tilde \eps_0$, for any $r \leq 1$, $\alpha \in (0,1)$, and a lower semi-continuous function $u:\overline{B_{2r}} \to \R$ such that
\begin{align*}
& u \geq 0 \text{ in } B_{2r}, \\
& M^-(D^2u, \grad u) \leq \eps_1 \text{ in } B_{2r}, \\
& |\{ u > r^\alpha \} \cap B_{r} | \geq \frac 12 |B_{r}|,
\end{align*}
then $u > \eps_1 r^\alpha$ in $B_r$.
\end{lemma}
\begin{remark}
As we shall see when proving this lemma, $\tilde \eps_0 = \eps_0
\eps_1$ where $\eps_0$ is given by Lemma~\ref{t:Lepsilon}.
\end{remark}
\begin{proof}
Let $\tau$ be the universal constant such that $C \tau^{-\eps} < |B_1|/2$, where $C$ and $\eps$ are the constants of Theorem~\ref{t:Lepsilon}. Consider the function
$\tilde u(x) = \tau r^{-\alpha} u(rx)$. It satisfies the properties
\begin{align*}
& \tilde u \geq 0 \text{ in } B_2, \\
& M^-(D^2 \tilde u, \grad \tilde u) \leq \tau r^{2-\alpha} \eps_1 \text{ in } B_2, \\
& |\{ \tilde u > \tau \} \cap B_1 | \geq \frac 12 |B_1| > C \tau^{-\eps},
\end{align*}
with $\tau r^{1-\alpha} \ts$ instead of $\ts$.

Let us choose $\eps_1 = \tau^{-1}$. Since
$r \leq 1$, we have
\[M^-(D^2 \tilde u, \grad \tilde u) \leq 1 \text{ in } B_2.\]
We now apply Theorem \ref{t:Lepsilon} and obtain that $\tilde u > 1$
in $B_1$ provided that $\tau r^{1-\alpha} \ts \leq \eps_0$. We just
have to choose $\tilde \eps_0 = \eps_0 \tau^{-1}= \eps_0 \eps_1$ since
$r^{1-\alpha} \leq 1$. Scaling back, we obtain $u > \eps_1 r^\alpha$
in $B_r$.
\end{proof}

\section{H\"older continuity}
\label{sec:holder}

In this section, we derive the H\"older estimates of Theorem~\ref{t:holder} 
from the (scaled) $L^\eps$ estimate.
\begin{proof}[Proof of Theorem~\ref{t:holder}]
We start by normalizing the solution $u$. Let
\[v(x) = \frac{u(\rho x)}{C_0(1+\eps_1^{-1})},\]
where $\rho \leq 1$ and $\eps_1$ is the constant from Lemma \ref{l:improvementOfOscillation}. The function $v$ satisfies the estimates
\begin{align*}
& M^-(D^2v, \grad v) \leq \eps_1 \text{ in } B_1, \\
& M^+(D^2v, \grad v) \geq -\eps_1 \text{ in } B_1, \\
& \|v\|_{L^\infty(B_1)} \leq 1,
\end{align*}
with $\ts$ replaced by 
$\frac{\rho}{C_0(1+\eps_1^{-1})} \ts$. Thus, we pick $\rho \leq 1$ such that
\[ \frac{\rho \ts}{C_0(1+\eps_1^{-1})} \leq \tilde \eps_0,\]
where $\tilde \eps_0$ is given by Lemma~\ref{l:improvementOfOscillation}.
It is enough to pick 
\[ \rho = \min \left(1, \frac{\tilde \eps_0 C_0(1+\eps_1^{-1})}{\ts} \right).\]

Let $a_k = \min_{B_{2^{-k}}} v$ and $b_k = \max_{B_{2^{-k}}} v$. We will prove that for some $\alpha>0$, 
\begin{equation} \label{e:oscillation} b_k - a_k \leq 2 \times 2^{-\alpha k}.
\end{equation}
 For $k=0$, the statement is obvious since $b_0 \leq \|v\|_{L^\infty(B_1)}$ and $a_0 \geq -\|v\|_{L^\infty(B_1)}$, thus $b_0 - a_0 \leq 2$. Now we proceed by induction.

Assume that $b_k - a_k \leq 2 \times 2^{-\alpha k}$ and let us prove
that $b_{k+1}-a_{k+1} \leq 2 \times 2^{-\alpha (k+1)}$. If $b_k - a_k
\leq 2 \times 2^{-\alpha (k+1)}$, then we are done since
$b_{k+1}-a_{k+1} \leq b_k -a_k$. Hence, we can assume that 
$\frac{b_k-a_k}2 \ge 2^{-\alpha (k+1)}$.

Let $m_k = (a_k+b_k)/2$. We have two alternatives. Either $|\{v>m_k\} \cap B_{2^{-k-1}}| \geq |B_{2^{-k-1}}|/2$ or $|\{v \leq m_k\} \cap B_{2^{-k-1}}| \geq |B_{2^{-k-1}}|/2$. In the first case we will prove that $a_{k+1}$ is larger than $a_k$, whereas in the second case we will show that $b_{k+1}$ is smaller than $b_k$.

Let us assume the first case, i.e. $|\{v>m_k\} \cap B_{2^{-k-1}}| \geq |B_{2^{-k-1}}|/2$. We apply Lemma \ref{l:improvementOfOscillation} to $v - a_k$ with $r = 2^{-k-1}$ to obtain that $v - a_k \geq {\eps_1} 2^{-(k+1)\alpha}$ for some ${\eps_1} > 0$ universal. Therefore, we have that $a_{k+1} \geq a_k + {\eps_1} 2^{-(k+1)\alpha}$. In particular
\[ b_{k+1} - a_{k+1} \leq b_k - a_k - {\eps_1} 2^{-(k+1)\alpha} \leq
(2^{\alpha+1}-{\eps_1})2^{-(k+1)\alpha} \leq 2 \times 2^{-(k+1) \alpha} \]
as soon as $\alpha$ is chosen  small enough so that $2^{\alpha+1} \leq 2+{\eps_1}$.

The estimate \eqref{e:oscillation} implies that $v$ is
$C^\alpha$ at the origin, with
\[ |v(x) - v(0)| \leq 4 |x|^\alpha,\]
for all $x \in B_1$. Scaling back to the function $u$, this means that
for all $x \in B_{\rho}$, 
\begin{align*}
 |u(x) - u(0)| & \leq 4 (1+\eps_1^{-1}) \rho^{-\alpha} C_0 |x|^\alpha \\
& \leq C C_0 |x|^\alpha
\end{align*}
where $C = C(\lambda,\Lambda,d,\gamma/C_0)$. By a standard translation
and covering argument, we have that $u \in C^\alpha(B_{1/2})$ and
\begin{align*}
  [u]_{C^{0,\alpha} (B_{1/2})} & \leq \tilde{C} C_0
\end{align*}
where $\tilde{C}$ differs from $C$ by a universal constant The proof
is now complete.
\end{proof}

\section{Harnack inequality}
\label{sec:harnack}

This section is devoted to the derivation of a Harnack inequality.
\begin{proof}[Proof of Theorem~\ref{t:harnack}]
We first reduce the problem to $C_0=1$ and $\inf_{B_{1/2}} u \le 1$ by
replacing $u$ with $u / (C_0+ \inf_{B_{1/2}}u)$. In particular,
$\gamma$ is replaced with $\gamma/ (C_0 + \inf_{B_{1/2}}u)$.

Let $\beta >0$ and let $h_t(x) = t(3/4-|x|)^{-\beta}$ be
defined in $B_{3/4}$. We consider the minimum value of $t$ such that
$h_t \geq u$ in $B_{3/4}$. The objective of the proof is to show that
this value of $t$ cannot be too large. If $t \le 1$, we are
done. Hence, we further assume that $t \ge 1$. 

Since $t$ is chosen to be the \emph{minimum} value such that $h_t \geq
u$, then there must exist some $x_0 \in B_1$ such that $h_t(x_0) =
u(x_0)$. Let $r = (3/4-|x_0|)/2$. That is, $2r$ is the distance from
$x_0$ and $\partial B_{3/4}$. Let $H_0 := h_t(x_0) = t(2r)^{-\beta}
\ge 1$.

We will estimate the measure of the set $\{u \geq H_0/2\} \cap
B_r(x_0)$ in two different ways. We will get a contradiction if $t$ is
too large.

Let us start by an upper bound of the measure. From Theorem
\ref{t:Lepsilon}, properly rescaled,
\begin{equation} \label{e:smallproportionabove}  
|\{u>H_0\} \cap B_r(x_0)| \leq |\{u>H_0\} \cap B_{3/4}| \leq
H_0^{-\eps} = C t^{-\eps} (2r)^{\beta \eps}.
\end{equation}

Let us now obtain a lower bound. Let $\mu$ be the small universal
constant and $\beta$ be a large universal constant such that
\begin{align*}
 M \left( \left(\frac{2-\mu}{2}\right)^{-\beta} - 1 \right) & \leq
 \frac 12 \\
\frac{(\mu r)^2}{\left( \left(\frac{2-\mu}{2}\right)^{-\beta} - 1
  \right)} & \leq 1, \\
\frac{(\mu r) \ts}{\left(
  \left(\frac{2-\mu}{2}\right)^{-\beta} - 1 \right)} &\leq \eps_0 \\
\beta & \geq \frac{d}{\eps}.
\end{align*}
where $M$ and $\eps_0$ are constants from Corollary \ref{c:weakweakharnack}
and $\eps$ comes from Theorem~\ref{t:Lepsilon}. The reader can check
that choosing 
\[ \beta = \eps^{-1} \max (d,\gamma) \]
and $\mu$ small enough
so that 
\[ \mu \le \frac{\gamma}{\eps_0} \quad \text{and} \quad 
\frac{\ln (1+\mu \gamma /\eps)}{-\ln(1-\mu/2)} \le \frac\gamma\eps
\quad \text{ and } \quad (1-\mu/2)^{-\beta} -1 \le \frac1{2M},\]
we get the four desired inequalities.

The maximum of $u$ in the ball $B_{\mu r}(x_0)$ is at most the maximum
of $h_t$ which equals $t(2r-\mu r)^{-\beta} =
\left(\frac{2-\mu}{2}\right)^{-\beta} H_0$. Let us define the
function \[ v(x) = \frac{\left(\frac{2-\mu}{2}\right)^{-\beta} H_0
  -u(x_0+\mu r x)}{\left( \left(\frac{2-\mu}{2}\right)^{-\beta} - 1
  \right) H_0}. \] Note that $v(0) = 1$, $v$ is non-negative in $B_1$
and satisfies the equations
\begin{align*}
& M^-(D^2v, \grad v) \leq \frac{(\mu r)^2}{\left( \left(\frac{2-\mu}{2}\right)^{-\beta} - 1 \right) H_0} \leq 1 \text{ in } B_1, \\
& M^+(D^2v, \grad v) \geq -\frac{(\mu r)^2}{\left( \left(\frac{2-\mu}{2}\right)^{-\beta} - 1 \right) H_0} \geq -1 \text{ in } B_1
\end{align*}
with $\ts$ replaced by $\frac{(\mu r) \ts}{H_0 \left(
  \left(\frac{2-\mu}{2}\right)^{-\beta} - 1 \right)} \leq \eps_0$
(because of the choice of $\mu$ and $\beta$).

We can apply Corollary
\ref{c:weakweakharnack} (in fact, its counterpositive) and obtain
\[ |\{v \leq M\} \cap B_{1/2}| \geq \delta |B_{1/2}|. \]

In terms of the original function $u$, this is an estimate of a set
where $u$ is larger than
\[ H_0 \left(\left(\frac{2-\mu}{2}\right)^{-\beta} - M 
\left( \left(\frac{2-\mu}{2}\right)^{-\beta} - 1 \right) \right) \geq
\frac{H_0}2, \]
because of the choice of $\mu$ and $\beta$. Thus, we obtain the estimate
\[ |\{u \geq H_0/2\} \cap B_{\mu r}(x_0)| \geq \delta |B_{\mu r}|.\]
Together with \eqref{e:smallproportionabove}, this implies that $t$ is
bounded from above (using the fact that $\beta \geq d /\eps$).
\end{proof}

\section*{Acknowledgments}
Cyril Imbert was partially supported  by ANR grant ANR-12-BS01-0008-01.
Luis Silvestre was partially supported by NSF grants DMS-1254332 and
DMS-1065979.

\end{document}